\documentclass[11pt]{amsart}
 \usepackage{amsthm}
 \usepackage{amsmath}
 \usepackage{amsfonts}
 \usepackage{amssymb}
 \usepackage{amscd}
\usepackage[dvipsnames]{xcolor}
\usepackage{url}
\usepackage{tikz} 
\usetikzlibrary{arrows,shapes}
\usepackage{multirow}
\usepackage{longtable}

\usepackage{float}
\usepackage[caption=false]{subfig}

\usepackage[ruled,vlined]{algorithm2e}

\newtheorem{thm}{Theorem}[section]

\newtheorem{problem}[thm]{Problem}
\theoremstyle{definition}

\newtheorem{lem}[thm]{Lemma}

\theoremstyle{definition}
\newtheorem{defn}[thm]{Definition}
\newtheorem{rem}[thm]{Remark}

\numberwithin{equation}{section}

\def \R { {\mathbb R} }

\newcommand{\demph}[1]{\emph{\color{Blue}#1}}

\newcommand{\up}[1]{ {\color{Magenta}#1} }
\newcommand{\dwn}[1]{ {\color{NavyBlue}#1} }
\newcommand{\hz}[1]{ {\color{OliveGreen}#1} }
\newcommand{\circl}[1]{\raisebox{.5pt}{\textcircled{\raisebox{-.9pt} {#1}}}}

\def \alf { {\color{Magenta}\alpha} }
\def \bet { {\color{NavyBlue}\beta} }
\def \gam { {\color{OliveGreen}\gamma} }

 \title{Combinatorial and computational investigations of Neighbor-Joining bias}

\author[Davidson]{Ruth Davidson}
\address{
Ruth Davidson\\
Bellingham, WA, U.S.A.
}
\email{ruth.davidson.math@gmail.com}
\urladdr{}

\author[Mart\'in del Campo]{Abraham Mart\'in del Campo}
\address{Abraham Mart\'in del Campo\\
         Centro de Investigaci\'on en Matem\'aticas, A.C.\\
         Jalisco S/N, Col. Valenciana\\
         36023, Guanajuato, Gto.  \\
         M\'exico}
\email{abraham.mc@cimat.mx}
\urladdr{http://personal.cimat.mx:8181/~{}abraham.mc/}

\thanks{Work of Mart\'in del Campo supported by CONACYT under grant
A1-S-30035 and C\'atedras-1076}
\thanks{Ruth Davidson was partially supported by U.S. N.S.F. grant DMS-1401591}

\begin{document}
\pgfdeclarelayer{background}
\pgfdeclarelayer{foreground}
\pgfsetlayers{background,main,foreground}

\begin{abstract}
The Neighbor-Joining algorithm is a popular distance-based phylogenetic method
that computes a tree metric from a dissimilarity map arising from biological data.
Realizing dissimilarity maps as points in Euclidean space, the 
algorithm 
partitions the input space into polyhedral regions indexed by the combinatorial 
type of the trees returned.
A full combinatorial description of these regions has not been found yet; different sequences of Neighbor-Joining agglomeration events can produce the same combinatorial tree, therefore associating multiple geometric regions to the same algorithmic output. 
We resolve this confusion by defining agglomeration orders on trees, leading to a bijection between distinct regions of the output space and weighted Motzkin paths. 
As a result, we give a formula for the number of polyhedral regions depending only on the number of taxa. We conclude with a computational comparison between these polyhedral regions, to unveil biases introduced in any implementation of the algorithm.

\end{abstract}

\maketitle

\section{Introduction}

The Neighbor-Joining (NJ) algorithm~\cite{NJ87} is a polynomial-time phylogenetic tree construction method. 
It is \demph{agglomerative}, so it constructs ancestral relationships between taxa by clustering the most closely related taxa at each step until a complete phylogeny is formed.  
The performance of the NJ algorithm has been studied from multiple mathematical and biological perspectives~\cite{Bry05, EHLY08, EY07, GS06}. Moreover, some statistical conditions have been given to guarantee a good performance of the algorithm for different types of biological data~\cite{Att97, MLP09}.
Due to its speed and theoretical performance guarantees, NJ remains a popular tool in the design of phylogenetic pipelines for large and complex data sets \cite{SNPhylo, LiveNJ, NJst}.  

The NJ algorithm takes as input a dissimilarity map $D$ between $n$ taxa and returns an unrooted binary tree with edge weights forming a tree metric.  One established paradigm for evaluating the performance of the NJ algorithm is to consider it as a heuristic for either the Least Squares Phylogeny (LSP) problem, which is NP-complete \cite{Day87}, or the Balanced Minimum Evolution (BME) problem, which is a special case of a weighted LSP approach \cite{DG04}. 
The LSP problem asks for the tree metric $T$ that minimizes the function
$$
\sqrt{\sum_{{i, j} \in {[n] \choose 2}} (D_{i,j}-T_{i,j})^{2}}.
$$
This perspective gives rise to the problem of determining whether a given heuristic for LSP or BME is biased, favoring certain phylogenies over others in its output. 
Addressing this problem requires a clear accounting of the possible outputs of the NJ algorithm.

Dissimilarity maps are symmetric matrices that can be realized as points in a Euclidean space.
In this geometric setting, the space of all tree metrics
form a \demph{polyhedral fan}~\cite{EY07}, which is a union of polyhedral cones meeting along common faces. The geometry of the space of tree metrics is well-understood \cite{Phyl, SS04}.
Phylogenetic inference methods that take a dissimilarity map as an input divide the Euclidean space into a family of subsets indexed by the combinatorial type of the trees returned.
For the NJ algorithm, these subsets are also polyhedral regions, as its decision criteria
are linear inequalities on the input data.

Comparing these polyhedral cones had lead to new computational methods to evaluate the performance of the NJ algorithm as a heuristic for LSP in~\cite{DS14} and of NJ as a heuristic for BME in~\cite{EHLY08}.  These methods are based on measuring the spherical volume of the cones, which is the volume of their intersection with the unit sphere. 
The spherical volume of polyhedral cones had unveiled unexpected bias in the UPGMA method, an older agglomerative phylogenetic method that has poorer performance guarantees in comparison to NJ~\cite{DS13}. 

The contribution of this article is that we give a formula that enumerates the polyhedral cones in the partition returned by NJ algorithm, by giving a bijection between the cones and weighted Motzkin paths.  We also explore the inherent bias in the algorithm by computing the spherical fraction of the cones and noting significant variation between some cones, similar to those found in \cite{DS13} for the UPGMA algorithm.  

In Section~\ref{S:NJ} we give a in-depth description of NJ and provide related definitions necessary to describe distinct NJ outputs.  In Section~\ref{S:combinatoricsNJ} we give a bijection between weighted Motzkin paths and labeled Newick strings to describe the outputs combinatorially.  Lastly, in Section~\ref{S:volumes}  we estimate the spherical fraction of the NJ cones and unveil a bias. We propose two ways to correct the bias, and display the performance of these corrections via computational experiments for small numbers of taxa.

\section{The Neighbor-Joining Algorithm}\label{S:NJ}

A \demph{graph} is a pair of sets $\{V,E\}$ called vertices and edges, with the latter representing relations between pairs of vertices. 
The \demph{degree of a vertex} is the number of edges adjacent to it.
A \demph{path} in a graph is a sequence of edges joining a sequence of vertices without repetition of vertices, except for possibly the start and end vertices. 
If the sequence starts and ends at the same vertex, the path is called a \demph{cycle}.
A graph without cycles is a \demph{tree}.
The vertices of a tree are usually called \demph{nodes}, except for those of degree one which are called \demph{leaves}.
In trees, there is a unique path between any two vertices.
A tree is \demph{rooted} if there is a distinguished vertex $\rho$ called the \demph{root}.
If $T$ is a rooted tree, there is a partial order on the vertices of $T$ induced by the root $\rho$ and the unique path between vertices, such that $\rho \leq v$ for all vertices $v$ of $T$. A \demph{binary tree} is a tree where all vertices have degree three except for the leaves and the root, if a root exists.
We will restrict our attention to unrooted binary trees.  An unrooted binary tree with $n$ leaves has $n{-}2$ nodes and $2n{-}3$ edges. 
We also consider the \demph{star tree}, a non-binary unrooted tree with exactly one node that is adjacent to all leaves. We illustrate a star tree in the left image of Figure~\ref{F:star&2ndstep}.

\subsection{Description of the algorithm}\label{SS:DescriptionNJ}

Let $[n]:= \{1, \ldots, n\}$. A \demph{dissimilarity map} is a function 
$D: [n] \times [n] \to \R_{\geq 0}$ such that $D(a,b) = D(b,a)$, $D(a,a) = 0$, and $D(a,b) \geq 0$ for all $a,b \in [n]$. 
A dissimilarity map is \demph{additive} or a \demph{tree metric}, if there exists a tree $T$ with edges $E$ with a weight function $w: E \to \R_{\geq 0}$ on $T$ so that $D(a,b)$ equals the
sum of the edge weights on the unique path from $a$ to $b$ in $T$.

The NJ algorithm takes a dissimilarity map as input and returns an unrooted binary tree with a tree metric. It is useful to view the progress of the NJ algorithm as a series of graph transformations. 
The algorithm starts from a star tree $t_n$, a graph with $n$ leaves and only one node ${\color{Blue}\mathcal{ O}}$ adjacent to all leaves, as illustrated on the left of  Figure~\ref{F:star&2ndstep}.  
Throughout the algorithm, the node $\mathcal O$ plays an important role; 
thus, we give special names to the vertices adjacent to $\mathcal O$. 
\begin{defn}\label{dfn:StemsAndBouquets}
  Vertices adjacent to $\mathcal O$ are called \demph{boughs}. Vertices that are boughs can be either leaves or internal nodes, so we refer to them as \demph{stems} and \demph{bouquets} respectively.
\end{defn}

In the recursive step, the algorithm takes a tree $t_{k}$ with $k$ boughs, it selects a pair of them and joins them by adding a new node adjacent to $\mathcal O$ in a way that the resulting tree $t_{k-1}$
has $k{-}1$ boughs. 
The algorithm iterates this step until there are only three boughs. 
Figure~\ref{F:star&2ndstep} illustrates the first two steps in the NJ algorithm, starting from the star $t_7$ and the 
next graph constructed by the NJ algorithm corresponding to the tree $t_{6}$, 
obtained from the star by joining the stems $a,b$ by introducing a bouquet $u$.
The subgraph consisting of only two adjacent leaves is sometimes called a \demph{cherry}, and the recursive step of the NJ algorithm is then referred as \demph{cherry picking} step~\cite{EY07}. This term served as inspiration for the names in Definition~\ref{dfn:StemsAndBouquets}, as it resembles the way the NJ algorithm creates trees. 

%
\begin{figure}[tb]
   \begin{tikzpicture}
        \node[label={[xshift=9pt, yshift=-13.5pt]$\mathcal O$}] (0) at (0,0) {$\bullet$};
        \node[label={[xshift=5.5pt, yshift=-6.5pt]$a$}] (2) at (0.7071,0.7071) {$\bullet$};
        \node (3) at (0,1) {$\bullet$};
        \node (4) at (-0.7071,0.7071) {$\bullet$};
        \node (5) at (-1,0) {$\bullet$};
        \node (6) at (-0.7071, -0.7071) {$\bullet$};
        \node (7) at (0,-1) {$\bullet$};
        \node[label={[xshift=5.5pt, yshift=-16.5pt]$b$}] (8) at (0.7071, -0.7071) {$\bullet$};
        \draw[Purple, ultra thick, shorten <=-7pt, shorten >=-8pt] (0)--(2);
        \draw[Purple, ultra thick, shorten <=-4.7pt, shorten >=-5pt] (0)--(3);
        \draw[Purple, ultra thick, shorten <=-7pt, shorten >=-7.5pt] (0)--(4);
        \draw[Purple, ultra thick, shorten <=-5pt, shorten >=-5pt] (0)--(5);
        \draw[Purple, ultra thick, shorten <=-7.5pt, shorten >=-7pt] (0)--(6);
        \draw[Purple, ultra thick, shorten <=-5pt, shorten >=-4.5pt] (0)--(7);
        \draw[Purple, ultra thick, shorten <=-7.5pt, shorten >=-7.3pt] (0)--(8);
   \end{tikzpicture}
\hspace{35pt}
   \begin{tikzpicture}
        \node[label={[xshift=6pt, yshift=-8pt]$\mathcal O$}] (0) at (0,0) {$\bullet$};
        \node[label={[xshift=7pt, yshift=-12pt]$u$}] (x) at (0.866025,0) {$\bullet$};
        \node[label={[xshift=5.5pt, yshift=-6.5pt]$a$}] (2) at (1.2071,0.7071) {$\bullet$};
        \node (3) at (0,1) {$\bullet$};
        \node (4) at (-0.866025,0.7071) {$\bullet$};
        \node (5) at (-1,0) {$\bullet$};
        \node (6) at (-0.866025, -0.7071) {$\bullet$};
        \node (7) at (0,-1) {$\bullet$};
        \node[label={[xshift=5.5pt, yshift=-16.5pt]$b$}] (8) at (1.2071,-0.7071) {$\bullet$};
        \draw[Purple, ultra thick, shorten <=-5pt, shorten >=-5pt] (0)--(x);
        \draw[Purple, ultra thick, shorten <=-4.9pt, shorten >=-5.45pt] (x)--(2);
        \draw[Purple, ultra thick, shorten <=-5.1pt, shorten >=-5.45pt] (x)--(8);
        \draw[Purple, ultra thick, shorten <=-7pt, shorten >=-7.5pt] (0)--(4);
        \draw[Purple, ultra thick, shorten <=-5pt, shorten >=-5pt] (0)--(5);
        \draw[Purple, ultra thick, shorten <=-7pt, shorten >=-7pt] (0)--(6);
        \draw[Purple, ultra thick, shorten <=-5pt, shorten >=-4.5pt] (0)--(7);
        \draw[Purple, ultra thick, shorten <=-4.6pt, shorten >=-5pt] (0)--(3);
   \end{tikzpicture}
    \caption{The trees corresponding to the first and second step in the NJ algorithm}
    \label{F:star&2ndstep}
\end{figure}
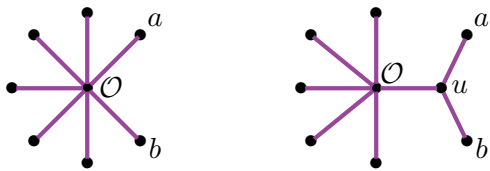

The algorithm selects a pair of boughs $a, b$ to agglomerate based on the $Q$-criterion, which is a function given by
\begin{equation}\label{Eq:Q_criterion}
    Q(a,b) = (k{-}2)D(a,b) - \sum_{c\in X} D(a,c) - \sum_{c\in X} D(b,c),
\end{equation}
where $X$ is the set of boughs, with $k=|X|$.
The NJ algorithm selects the vertices $a,b$ where $Q$ is minimal.
It has been shown that this criterion uniquely determines the NJ algorithm~\cite{Bry05}, and that
a pair of leaves minimizing
$Q$ come from a cherry in the true tree~\cite{NJ87, SK88}.
To agglomerate the selected nodes $a,b$, the NJ algorithm introduces a new bouquet $u$ adjacent to them, resulting in a tree with one bough less. It then estimates the distance in this new tree from the remaining vertices to $u$ via the reduction formula
\begin{equation}\label{Eq:ReductionStep}
    D(c,u) = \frac{1}{2}\left( D(a,c)+ D(b,c) - D(a,b) \right).
\end{equation}
%
We sometimes use $D_k$ and $Q_k$ to denote the dissimilarity map and the $Q$-criterion associated to the tree $t_k$ with
$k$ boughs. Thus, the last step of the algorithm is decided by a pair of nodes $a,b$ that minimize $Q_4$. However, in this last step there is more than one minimizing pair, as we demonstrate in the following lemma.

\begin{lem}\label{L:double_minima}
For $n = 4$, the matrix $Q_4$ always achieves its minimum in exactly two entries.
\end{lem}

\begin{proof}
Let $\{a,b,c, d\}$ be the four vertices corresponding to $t_4$ and suppose, without loss of generality, that the minimum in $Q_4$ is achieved in $Q(a,b)$.
Then, equation~\eqref{Eq:Q_criterion} writes as
\begin{align*}
Q(a,b) &= 2 D(a,b) - (D(a,b) + D(a,c) + D(a,d)) - (D(a,b) + D(b,c) + D(b,d)) \\
  &=  -(D(a,c) + D(a,d)) - (D(b,c) + D(b,d))\\
  &=  -(D(a,c) + D(b,c)) - (D(a,d) + D(b,d))\\
 &=  - (D(a,c) + D(b,c) + D(c,d)) - (D(a,d) + D(b,d) + D(c,d)) + 2 D(c,d) = Q(c,d).
\end{align*}
Therefore, we get that $Q(a,b) = Q(c,d)$, meaning that the minimum is achieved twice. 
\end{proof}

\begin{rem}
 The previous lemma is similar to the four-point condition but not the same,
 as Lemma~\ref{L:double_minima} holds even when the dissimilarity map $D$ is not additive. 
\end{rem}

\begin{rem}
Note that if $n>4$, the minimum in $Q_n$ could be achieved in more than one entry too, but that happens in a very small dimensional space in $\R^{n\choose 2}$, thus it occurs in a measure zero set. Only in the case $n = 4$ it happens always.
\end{rem}

\subsection{Newick notation}\label{SS:Newick}
We represent trees with the Newick notation.
This is one of the most widely used notations in bioinformatics to encode information about the tree topology, branch distances, and vertex labels.
It consist of parentheses that represent tree data as textual strings (see~\cite{War17,fels04}). 
A pair of vertices enclosed within matching parentheses indicates they have a common ancestor. There are slightly different formats representing the Newick notation,
so we explain the one we use with an example.

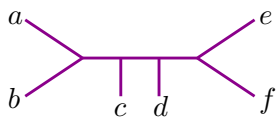
\begin{figure}[htb]
 \centerline{
  \begin{picture}(100,40)(0,0)
   \put(0, 0){\includegraphics[height=30pt]{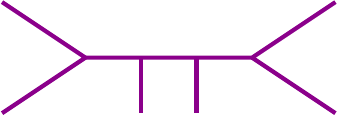}}
   \put(-6,28){$a$}      \put(-6,-4){$b$}
   \put(34,-7){$c$} \put(49,-7){$d$} 
   \put(89,28){$e$} \put(89,-4){$f$} 
  \end{picture}
 }
\caption{An unrooted binary tree $t$ with 6 taxa.} 
\label{F:NewickExample}
\end{figure}

Let $t$ be the tree in Figure~\ref{F:NewickExample}, with leaves labeled by the set $\{a,b,c,d,e,f\}$. Newick notation for $t$ can be written as $((c,(a,b)),(e,f), d)$, 
or $((d,(c, (a,b))), e, f)$, or in many other ways. This non-uniqueness of Newick strings can be inconvenient, as one must determine whether two strings represent
the same tree or not. 
We let the \demph{length} of a string in Newick notation to be the maximum number of parentheses of one orientation contained in the string (to the right or the left). For instance, the two strings above
have length 3.

We now explain the Newick notation in the NJ algorithm.
Let $[n]$ indicate the leaves of a tree.
The starting point is the start tree $t_n$, which is indicated
with the string $(1, 2,\ldots, n)$. 
Then, the algorithm selects
two vertices and we join them by enclosing them with a parenthesis, which
we append at the left to the remaining string.  
For instance, if $1$ and $2$ are the selected vertices, we would write $((1,2),3,\ldots,n)$ to indicate the tree obtained in the second step in the NJ algorithm. This is illustrated in the tree to the right of Figure~\ref{F:star&2ndstep}, by letting $a=1$ and $b=2$.
In the recursive step, NJ will take a string $s_k$ of length $k$ and write a string of length $k{-}1$ by joining two elements of $s_k$ with a new parenthesis attached to the left. In this way, in the string
$((c,(a,b)),(e,f), d)$ we know that the last step was to join $c$ with the pair $(a,b)$, but we do not know if the previous string was $((a,b),(e,f),c,d)$ or $((e,f),(a,b),c,d)$. A parenthesis in the Newick string indicates a node in the tree, so we could remove this ambiguity by labeling the nodes when we introduce them in the algorithm. 
We label these nodes with a circled number written to the left of a parenthesis, indicating the step when the node was created in the algorithm. 
For instance, we could write $(\circl{3}(c,\circl{1}(a,b)),\circl{2}(e,f), d)$ to indicate that the first step in the NJ algorithm was to join $(a,b)$, then $(e,f)$, and lastly $c$ with $(a,b)$.  We refer hereafter to this notation as \demph{ordered Newick notation}.
\newline

%
\begin{algorithm}[H]
\SetAlgoLined
\SetKwInOut{Input}{Input}\SetKwInOut{Output}{Output}
\Input{A dissimilarity map $D_n$ on the set $\{1, \ldots , n\}$.
}
\Output{An unrooted binary tree $T$ with leaf labels $\{1,  \ldots, n \}$ written in ordered Newick notation.}

{\bf Initialize} $k = n$, $r = 0$, and $S_k= (1,  \ldots,  k)$ representing the Newick format for 
$t_k$, 
the star tree on $k$ leaves with boughs $B_k$.

 \While{$k > 3$}{
 \begin{enumerate}
     \item Identify boughs $\{a, b\}$ of $t_k$ minimizing 
     $$Q_{k}(a,b) = (k-2)D_k(a, b) - \sum_{c \in B_k} D_k(a,c) - \sum_{c \in B_k} D_k(b,c).$$
     \item Define $S_{k-1}$ by appending \circl{$r$}$(a,b)$ to the left of $S_{k}$ after dropping $a$ and $b$.
     \item Construct $t_{k-1}$ from $t_{k}$ by 
     \begin{enumerate}
         \item Deleting the edges from both $a$ and $b$ to $\mathcal O$.
         \item Introducing a new vertex labeled $u$ adjacent to $\mathcal O$.
         \item Connecting $u$ to both $a$ and $b$. 
     \end{enumerate}
    \item Compute $D_{k-1}$ on the new set of boughs via the formula
     $$D_{k-1}(u,v) = \frac{1}{2}\left( D_{k}(a,v) + D_{k}(b,v) - D_{k}(A,B) \right)$$ for all remaining boughs in $t_{k}$
     \item Increase $r$ by one and decrease $k$ by 1.
 \end{enumerate}

\Return Label $\mathcal{O}$ in $t_3$ with zero and return $t_3   = T$
  
 }
 \caption{The Neighbor-Joining Algorithm}
\end{algorithm}

\section{Combinatorial description of NJ}\label{S:combinatoricsNJ}

To a given data matrix $D$, the NJ algorithm associates a binary tree with $n$ leaves, together with a tree metric. 
Without the distinction of ordered Newick notation, it can appear that the algorithm associates the same tree to different data.
For instance, let us consider the following matrices:
\begin{equation}\label{Eqn:DandD'Dissimilarity}
D = 
\begin{pmatrix}
0 & 3 & 5 & 4 & 7\\
3 & 0 & 10 & 3 & 7\\
5 & 10 & 0 & 6 & 5\\
4 & 3 & 6 & 0 & 2\\
7 & 7 & 5 & 2 & 0
\end{pmatrix},
\qquad 
D' = 
\begin{pmatrix}
 0 & 2 & 4 & 1 & 9 \\
 2 & 0 & 10 & 3 & 8 \\
 4 & 10 & 0 & 6 & 5 \\
 1 & 3 & 6 & 0 & 7 \\
 9 & 8 & 5 & 7 & 0
 \end{pmatrix}.
\end{equation}

The NJ algorithm associates the same unrooted binary tree to both of them.
However, their Newick notation is not the same.
For $D$, the corresponding Newick tree built by NJ is $((d,(a,b)), c, e)$
whereas for $D'$ we obtain $((d,(c,e)),a,b)$. We illustrate this in Figure~\ref{F:example1}.
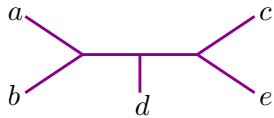
\begin{figure}[htb]
 \centerline{
  \begin{picture}(100,40)(0,0)
   \put(0, 0){\includegraphics[height=30pt]{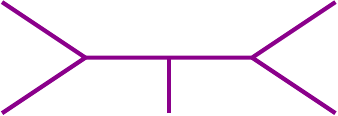}}
   \put(-6,28){$a$}      \put(-6,-4){$b$}
   \put(42,-8){$d$} 
   \put(89,28){$c$} \put(89,-4){$e$} 
  \end{picture}
 }
\caption{Tree corresponding to the Newick strings $((d,(a,b)), c, e)$ and $((d,(c,e)),a,b)$.} 
\label{F:example1}
\end{figure}
These two Newick strings encode one topological type of tree.  Yet they have a subtle difference 
if one needs to take into account the order in which taxa was agglomerated to interpret the output of the NJ algorithm, as one must do when establishing which polyhedral region of $\mathbb{R}^{n \choose 2}$ contains the dissimilarity matrix input that produces a given Newick string under NJ.  See Section \ref{S:volumes} for a discussion of the implications of locating the correspondence between these Newick strings and the geometry of NJ. 
For example, the Newick string $((d,(a,b)), c, e)$ corresponding to $D$ indicates
that $a, b$ were joined together and that $c$ and $e$ were never joined to anything, whereas the Newick $((d,(c,e)),a,b)$ for $D'$ indicates the opposite, that $a, b$ were never joined to anything but $c$ and $e$ were joined to each other.
%
%
To distinguish these data and relate them to ordered Newick notation, we endow binary trees with something we called \emph{agglomeration order} and we explain it next. 
 
\subsection{Binary trees with agglomeration order}\label{SS:agglomeration}

 The expected number of trees that arise as an output depends on the shape (topology) and the labels of the leaves. 
For unrooted trees with $n$ taxa, there are $(2(n {-} 2) {-}1)!!$ labeled binary trees.  Thus, for $n= 4$ there is only one tree shape and three ways to label the leaves.
However, there are two possible ways to write them in Newick notation.
For instance, we could have $((a,b),c,d)$ or $((c,d),a,b)$ to mean the same tree, but the NJ algorithm differentiates these two. Therefore, for $n=4$ 
there are 6 possible Newick strings but the NJ algorithm only returns 3 of them.
We summarize this information for the first five cases in the following table.
\begin{table}[htb]

\begin{tabular}{|c|c|c|c|c|}
\hline
Taxa & tree topologies & unrooted binary trees & trees from NJ & ordered Newick strings \\\hline 
4 & 1  & $3=3!!$ &3 & 6 \\\hline 
5 & 1  & $15=5!!$ & 30& 60 \\\hline 
6 & 2  & $105=7!!$ & 450 &900\\\hline 
7 & 2  & $945=9!!$ & 9,450 &18,900\\\hline 
8 & 4  & $10,395=11!!$ & 264,600 &529,200 \\\hline
\end{tabular}
\caption{Number of trees for small number of taxa.}
\label{Tb:counting trees}
\end{table}

We will give a combinatorial formula to compute the expected number of output trees from the algorithm. For this, we set the combinatorial definitions we will require.

\begin{defn}\label{dfn:agglomeration_order}
For a binary tree with $n$ leaves, an \demph{agglomeration order} means labeling the $n{-}2$ internal nodes with the set $\{\infty,1,2,\ldots,n{-}3\}$, such that
the labels of the internal nodes in every path from $\infty$ to a leaf form a decreasing sequence.
\end{defn}

We use circled numbers to indicate the assignment of the agglomeration order, and to simplify the schematics, we omit the label $\infty$, so it is easier to verify that all sequences of internal vertices starting there and terminating at a leaf are decreasing.
In Figure~\ref{F:AgglomOrder} we illustrate
an agglomeration order for the tree $t$ with 6 taxa
from Figure~\ref{F:NewickExample}. 
\begin{figure}[htb]
 \centerline{
  \begin{picture}(100,40)(0,0)
   \put(0, 0){\includegraphics[height=30pt]{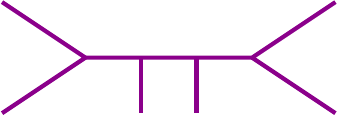}}
   \put(17,20){{\color{Purple}\circl{1}}}
   \put(31,20){{\color{Purple}\circl{3}}}
   \put(59,20){{\color{Purple}\circl{2}}}
   \put(-6,28){$a$}      \put(-6,-4){$b$}
   \put(34,-7){$c$} \put(49,-7){$d$} 
   \put(89,28){$e$} \put(89,-4){$f$} 
  \end{picture}
 }
\caption{An agglomeration order in an unrooted tree with 6 taxa.} 
\label{F:AgglomOrder}
\end{figure}
There, if we exchange  \circl{3} with \circl{1} we get a label of the internal nodes that is not an agglomeration order. 
In Newick notation, the tree in Figure~\ref{F:AgglomOrder} would be $(\circl{3}(c,\circl{1}(a,b)),\circl{2}(e,f), d)$.

\begin{rem}\label{rem:NJ2to1}
From Lemma~\ref{L:double_minima}, we see that the NJ algorithm associates two ordered trees to a single data matrix $D$, as it has to decide between two agglomeration orders in the last step.
\end{rem}

\begin{thm}\label{T:NJagglomeration}
The set of 
agglomeration orders on unrooted binary trees is in a 2-to-1 correspondence with the 
output space of the NJ algorithm. 
\end{thm}

\begin{proof}
The NJ algorithm starts with the star tree $t_n$ consisting of $n$ leaves and just one internal node $\mathcal O$. From there, at each step, the algorithm applies a series of graph transformations to a given tree, preserving the number of leaves, while increasing the number of nodes by one. 
This new node is adjacent to $\mathcal O$. Thus, at each step the new node is closer (combinatorially, not metrically) to $\mathcal O$ than all other nodes in a path to a leaf.
Numbering each node with the moment they appear in the algorithm, starting with $\infty$ for the node $\mathcal O$, results in an agglomeration order. As these steps are reversible, each agglomeration order can arise from the algorithm, giving the 2-to-1 correspondence with the output space, together with Remark~\ref{rem:NJ2to1}.
\end{proof}

Therefore, understanding the NJ algorithm leads to understanding orders of agglomeration for binary trees. To simplify the rest of the exposition, we give the following definitions.

\begin{defn}\label{Dfn:AgglomeratedTree}
  We use the term \demph{agglomerated tree} to refer to  an unrooted binary tree endowed with an agglomeration order.  The number of agglomerated trees with $n$ leaves will be denoted by \demph{$\Phi(n)$}. 
\end{defn}

Computing the number  $\Phi(n)$ is important to understand more about the combinatorial complexity of the NJ algorithm. 
Due to Theorem~\ref{T:NJagglomeration} above, the number of trees output by the NJ algorithm is $\Phi(n)/2$.
The last column of Table~\ref{Tb:counting trees} is precisely the value of  $\Phi(n)$ for $n=4,\ldots, 8$. We tried to determine $\Phi(n)$ by estimating first the number of agglomeration orders, knowing that the number of unrooted binary trees is given by the combinatorial formula $(2n{-}5)!!$. 
For $n=4$ and $5$, there is only one tree topology, but the number of agglomeration orders they have is $2$ and $4$ respectively.
For these cases, it holds true that $\Phi(4)= 2\cdot 3!!$ and $\Phi(5)=4\cdot 5!!$. 
However, as we can see in Table~\ref{Tb:counting trees},  this is no longer the case for other values of $n$, as $\Phi(n)$ is not always divisible by $(2n{-5})!!$. 
Nonetheless, we were able to give a formula for the number $\Phi(n)$ using Motzkin paths. However, it remains open to understand more about the connection between unrooted binary trees and agglomeration orders.

\begin{problem}
Determine the number of agglomeration orders that can be assigned to a given unrooted binary tree.
\end{problem}

\subsection{Motzkin paths}

Motzkin paths are combinatorial structures appearing in many contexts. They are counted by Motzkin numbers, which are related to Catalan numbers~\cite{Aig98,Aig99,OVdJ15,MOPT10}.  Note that while Catalan numbers are known to count combinatorial objects referred to as \emph{planar rooted trees}~\cite{DSh02}, the trees in this paper are fundamentally different objects.

A \demph{Motzkin path} is an integer lattice path starting and ending in the horizontal axis without crossing it, consisting of up steps $\up{u}=\up{(1,1)}$, down steps $\dwn{d}=\dwn{(1,-1)}$, and horizontal steps $\hz{h}=\hz{(1,0)}$.
A Motzkin path with no horizontal steps is a \demph{Dyck path}. The number of Dyck paths from $(0,0)$ to $(2N,0)$ is given by the Catalan number $C_N$, whereas the number of Motzkin paths from $(0,0)$ to $(0,N)$,  is given by the Motzkin number $M_N$.

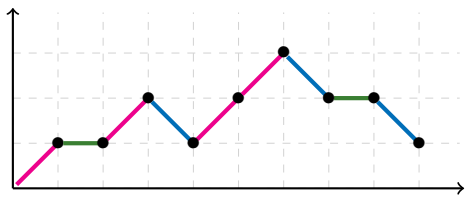
\begin{figure}[htb]
    \begin{tikzpicture}[thick, scale=0.6]
        \draw[help lines, color=gray!35, dashed] (0,0) grid (9.9,3.9);
        \draw[->,thick] (0,0)--(10,0);
        \draw[->,thick] (0,0)--(0,4);
        \draw[Magenta, ultra thick, shorten <=2pt] (0,0)--(1,1) node[black] {$\bullet$};
        \draw[Magenta, ultra thick, shorten <=2pt] (2,1)--(3,2) node[black] {$\bullet$};
        \draw[Magenta, ultra thick, shorten <=2pt] (4,1)--(5,2) node[black] {$\bullet$};
        \draw[Magenta, ultra thick,shorten <=2pt] (5,2)--(6,3) node[black] {$\bullet$};
        \draw[OliveGreen, ultra thick, shorten <=2pt] (1,1)--(2,1) node[black] {$\bullet$};
        \draw[OliveGreen, ultra thick, shorten <=2pt] (7,2)--(8,2) node[black] {$\bullet$};
        \draw[NavyBlue, ultra thick, shorten <=2pt] (3,2)--(4,1) node[black] {$\bullet$};
        \draw[NavyBlue, ultra thick, shorten <=2pt] (6,3)--(7,2) node[black] {$\bullet$};
        \draw[NavyBlue, ultra thick, shorten <=2pt] (8,2)--(9,1) node[black] {$\bullet$};
    \end{tikzpicture}
\caption{Motzkin path} 
\label{F:MotzkinPath}
\end{figure}

Our main result here is to give a bijection between Motzkin paths and agglomerated trees. This bijection allows the derivation of a formula to determine the number $\Phi(n)$ that counts 
the number of agglomerated trees.
The key is to focus on the recursive step of the algorithm.

\begin{rem}\label{Rem:AlphaSteps}
Let $\mathcal O$ be the unique node of the star tree $t_n$.
In the recursive step, the NJ algorithm takes a tree $t_k$ with 
$\ell$ stems and $r$ bouquets, such that $k=\ell+r$.
From there, it constructs the graph $t_{k-1}$ by adding an internal node 
in three possible ways:
\begin{itemize}
    \item It merges two stems. We call this step $\alf$. 
    \item  It merges two bouquets. We call this step $\bet$.
    \item It merges a stem to a bouquet. We call this step $\gam$.
\end{itemize}
We illustrate these three steps in Figure~\ref{F:AgglomerationSteps} below.
\begin{figure}[htb]
  \includegraphics[height=0.36\textwidth]{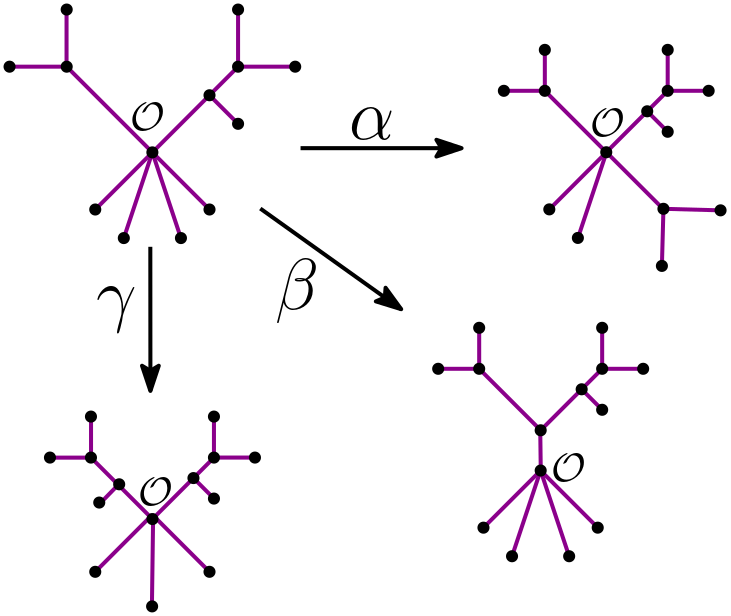}
\caption{The tree agglomeration steps.} 
\label{F:AgglomerationSteps}
\end{figure}
\end{rem}

From this remark, we see that
we can summarize a given tree $t$ with a distinguished node $\mathcal O$
by its \demph{bough vector $(\ell, r)$}, 
where $\ell$ and $r$ are the number of stems and bouquets in $t$ respectively.
For instance, the tree $t$ in Figure~\ref{F:TreeNotBinary} has three nodes and nine leaves, but just four stems and two bouquets. Thus, the bough vector $(4,2)$
summarizes $t$. 
\begin{figure}[htb]
  \includegraphics[height=0.16\textwidth]{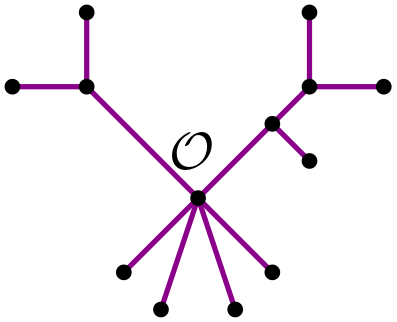}
\caption{A tree $t$ with bough vector $(4,2)$.} 
\label{F:TreeNotBinary}
\end{figure}
The bough vector for the star with $n$ leaves is $(n,0)$.
Let $t_k$ be a tree from the NJ algorithm, and let $(\ell, r)$ be its bough vector, so $k=\ell + r$. Note that the tree obtained from $t_k$ after an $\alf$ step is summarized by the bough vector $(l, r) + \up{(-2,1)}$.
Similarly, after a step $\bet$ or $\gam$, the bough vector of the tree is 
$(l, r) + \dwn{(0,-1)}$, or $(l,r) + \hz{(-1, 0)}$ respectively.
Note that the NJ algorithm ends with a tree $T_3$ consisting only of three boughs.
Thus, for $n\geq 4$, the possible bough vectors for the last step are $(2,1), (1,2)$, or $(0,3)$.
Note that the first step of the algorithm is forced to be always an $\alf$ step, thus we can omit it and analyze the rest of the steps, starting at $(n{-}2,1)$.

\begin{defn}
  For $n\geq 4$ we define an \demph{NJ path} of length $n{-}4$ as an integer lattice path consisting of steps   $\alf =  \up{(-2,1)}$,  $\bet = \dwn{(0,-1)}$, and $\gam= \hz{(-1, 0)}$, starting at $(n{-}2,1)$ and ending at one of $(2,1), (1,2)$, or $(0,3)$, without crossing the horizontal axis $y=1$.
\end{defn}

\begin{thm}\label{Thm:NJ-MotzkinPaths}
  For every $n\geq 4$, there is a bijection between NJ paths of length $n{-}4$ and Motzkin paths of length $n{-}4$ from $(1,1)$ to either $(n{-}3,1), (n{-}3,2)$, or $(n{-}3,3)$.
\end{thm}

\begin{proof}
The matrix
$\big(\begin{smallmatrix}
  -1 & -1\\
  0 & 1
\end{smallmatrix}\big)$
sends steps $\{\alf, \bet, \gam\}$ into $\{\up{u}, \dwn{d}, \hz{h}\}$ respectively, giving the  bijection. 
\end{proof}

Motzkin paths starting at the origin and ending at $(\ell,r)$ are called \demph{partial Motzkin paths}. Thus, after translating $(1,1) \to (0,0)$, we could write Theorem~\ref{Thm:NJ-MotzkinPaths} in terms of partial Motzkin paths ending at $(n{-}4,0), (n{-}4,1)$, or $(n{-}4,2)$.
%
%
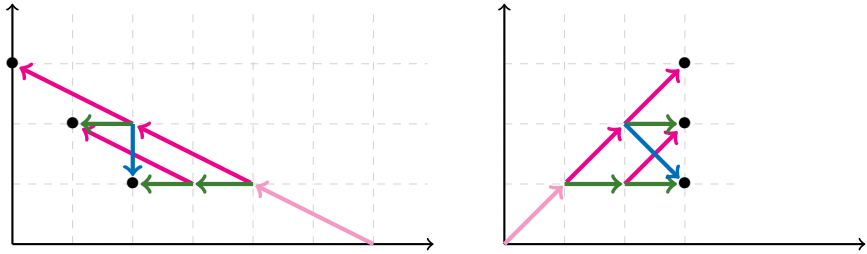
\begin{figure}[htb!]
    \begin{tikzpicture}[thick, scale=0.8]
        \draw[help lines, color=gray!35, dashed] (0,0) grid (6.9,3.9);
        \draw[->,thick] (0,0)--(7,0);
        \draw[->,thick] (0,0)--(0,4);
        \node (A) at (2,1) {$\bullet$};
        \node (B) at (1,2) {$\bullet$};
        \node (C) at (0,3) {$\bullet$};
        
        \draw[->,Magenta!50, ultra thick,shorten >=1pt] (6,0)--(4,1);
        \draw[->,Magenta, ultra thick,shorten >=2pt,shorten <=1pt] (4,1)--(2,2);
        \draw[->,Magenta, ultra thick,shorten >=3pt] (2,2)--(0,3);
        \draw[->,Magenta, ultra thick,shorten >=4pt] (3,1)--(1,2);
        \draw[->, OliveGreen, ultra thick,shorten >=1pt] (4,1)--(3,1);
        \draw[->, OliveGreen, ultra thick,shorten >=3pt] (3,1)--(2,1);
        \draw[->, OliveGreen, ultra thick,shorten >=3pt] (2,2)--(1,2);
        \draw[->, NavyBlue, ultra thick,shorten >=3pt] (2,2)--(2,1);
    \end{tikzpicture}
    \qquad
    \begin{tikzpicture}[thick, scale=0.8]
        \draw[help lines, color=gray!35, dashed] (0,0) grid (3.9,3.9);
        \draw[->,thick] (0,0)--(6,0);
        \draw[->,thick] (0,0)--(0,4);
        \node (A) at (3,1) {$\bullet$};
        \node (B) at (3,2) {$\bullet$};
        \node (C) at (3,3) {$\bullet$};
        
        \draw[->,Magenta!50, ultra thick,shorten >=1pt] (0,0)--(1,1);
        \draw[->,Magenta, ultra thick,shorten >=2pt,shorten <=1pt] (1,1)--(2,2);
        \draw[->,Magenta, ultra thick,shorten >=3pt] (2,2)--(3,3);
         \draw[->,Magenta, ultra thick,shorten >=4pt] (2,1)--(3,2);
         \draw[->, OliveGreen, ultra thick,shorten >=1pt] (1,1)--(2,1);
         \draw[->, OliveGreen, ultra thick,shorten >=3pt] (2,1)--(3,1);
         \draw[->, OliveGreen, ultra thick,shorten >=3pt] (2,2)--(3,2);
         \draw[->, NavyBlue, ultra thick,shorten >=3pt] (2,2)--(3,1);
     \end{tikzpicture}
    \caption{All NJ paths (left) and partial Motzkin paths (right) for 6 taxa.}
    \label{F:NJpaths6}
\end{figure}

For 6 taxa, all possible NJ paths and their corresponding Motzkin paths are depicted in Figure~\ref{F:NJpaths6}. NJ paths there start at $(6,0)$ corresponding to the start tree $t_6$, and the first step is always an $\alf$ step. From there, one chooses from the three steps
$\{\alf, \bet, \gam\}$ consecutively until reaching one of the points $(0,3), (1,2)$, or $(2,1)$. In this case, there are only 5 paths, each corresponding to an agglomeration order 
in Figure~\ref{F:AllOrders6}. 
\begin{figure}[htb]
\subfloat{
 \begin{picture}(80,55)(0,0)
   \put(0, 0){\includegraphics[height=60pt]{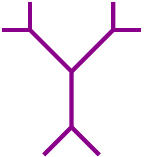}}
   \put(41,38){{\color{Purple}\circl{2}}}
   \put(3,38){{\color{Purple}\circl{1}}}
   \put(15,11){{\color{Purple}\circl{3}}}
  \end{picture}
} 
\subfloat{\begin{picture}(80,40)(0,0)
   \put(0, 0){\includegraphics[height=60pt]{images/nonCaterpillar6.pdf}}
   \put(41,38){{\color{Purple}\circl{2}}}
   \put(3,38){{\color{Purple}\circl{1}}}
   \put(15,27){{\color{Purple}\circl{3}}}
  \end{picture}
}
\\
\subfloat{
  \begin{picture}(100,40)(0,0)
   \put(0, 0){\includegraphics[height=30pt]{images/unrootedCaterpillar6.pdf}}
   \put(17,20){{\color{Purple}\circl{1}}}
   \put(31,20){{\color{Purple}\circl{2}}}
   \put(45,20){{\color{Purple}\circl{3}}}
  \end{picture}
} 
\subfloat{
  \begin{picture}(100,40)(0,0)
   \put(0, 0){\includegraphics[height=30pt]{images/unrootedCaterpillar6.pdf}}
   \put(17,20){{\color{Purple}\circl{1}}}
   \put(31,20){{\color{Purple}\circl{2}}}
   \put(59,20){{\color{Purple}\circl{3}}}
  \end{picture}
} 
\subfloat{
  \begin{picture}(100,40)(0,0)
   \put(0, 0){\includegraphics[height=30pt]{images/unrootedCaterpillar6.pdf}}
   \put(17,20){{\color{Purple}\circl{1}}}
   \put(31,20){{\color{Purple}\circl{3}}}
   \put(59,20){{\color{Purple}\circl{2}}}
  \end{picture}
}
\caption{Some agglomeration orders corresponding to the 5 Motzkin paths in Figure~\ref{F:NJpaths6}.} 
\label{F:AllOrders6}
\end{figure}
For instance, the path formed by the sequence $\alf\alf\gam$ (read from left to right) is in correspondence to the agglomeration order of the tree in Figure~\ref{F:AgglomOrder}, denoted in Newick format by $(\hz{\circl{3}}(c,\up{\circl{1}}(a,b)),\up{\circl{2}}(e,f),d)$. 
Note that the sequence $\alf\alf\gam$ is in correspondence with more than one agglomeration order. For instance, it is also in correspondence with the tree $(\hz{\circl{3}}(c,\up{\circl{1}}(e,f)),\up{\circl{2}}(a,b),d)$. 


\subsection{The number of agglomerated trees}


The results of the previous discussion reduce the counting of the number of trees obtained from the NJ algorithm to enumerating some partial Motzkin paths, which are counted by Motzkin numbers.
Hence, we let $M_k$ be the number of partial Motzkin paths of length $k$. More specifically, we let $M_{k,j}$ be the number of partial Motzkin paths on length $k$ that end at level $j$. 
It is known (see~\cite{OVdJ15}) that the Motzkin numbers $M_k$ 
satisfy the following formulas involving Catalan numbers:
\begin{align}\label{Eq:MotzkinCatalan}
    M_k = \sum_{i} {k\choose 2i}C_i, \qquad\mbox{and}\qquad C_{k+1} = \sum_i {k\choose i}M_i.
\end{align}
While the numbers $M_{k,j}$ satisfy the following formula~\cite[Theorem 10.8.1]{Bon15}
\begin{equation}\label{E:Motzkin_nk}
M_{k,j} = \sum_{i=0}^k {k \choose i}\left[{k{-}i \choose (k{+}j{-}i)/2} - {k{-}i \choose (k{+}j{-}i{+}2)/2}\right],
\end{equation}
where, by convention, a binomial coefficient is 0 if its bottom parameter is not an integer.
These numbers $M_{k,j}$ construct the \demph{Motzkin triangle} that enumerates all partial Motzkin paths~\cite{Lan03}, which is shown in Figure~\ref{F:MotzkinTriangle}. The triangle is defined recursively by
\begin{equation}\label{Eq:MotzkinTriangleDefn}
    M_{0,0} = 1, \qquad M_{k+1,j} = M_{k,j-1} + M_{k,j} + M_{k,j+1}, \mbox{ for }\ k\geq 1.
\end{equation}
%
\begin{figure}[h!tb]
\begin{tikzpicture}[scale=0.8]
    \node  [circle,draw] (00) at (0,0) {1};
    \node  [circle,draw] (10) at (1.5,0) {1};
    \node  [circle,draw] (20) at (3,0) {2};
    \node  [circle,draw] (30) at (4.5,0) {4};
    \node  [circle,draw] (40) at (6,0) {9};
    \node  [circle,draw] (50) at (7.5,0) {21};
    \node  [circle,draw] (11) at (1.5,1.5) {1};
    \node  [circle,draw] (21) at (3,1.5) {2};
    \node  [circle,draw] (31) at (4.5,1.5) {5};
    \node  [circle,draw] (41) at (6,1.5) {12};
    \node  [circle,draw] (51) at (7.5,1.5) {30};
    \node  [circle,draw] (22) at (3,3) {1};
    \node  [circle,draw] (32) at (4.5,3) {3};
    \node  [circle,draw] (42) at (6,3) {9};
    \node  [circle,draw] (52) at (7.5,3) {25};
    \node  [circle,draw] (33) at (4.5,4.5) {1};
    \node  [circle,draw] (43) at (6,4.5) {4};
    \node  [circle,draw] (53) at (7.5,4.5) {14};
    \node  [circle,draw] (44) at (6,6) {1};
    \node  [circle,draw] (54) at (7.5,6) {5};
    \node  [circle,draw] (55) at (7.5,7.5) {1};
    \draw [->, OliveGreen,  thick] (00) -- (10);
    \draw [->, OliveGreen,  thick] (10) -- (20);
    \draw [->, OliveGreen,  thick] (20) -- (30);
    \draw [->, OliveGreen,  thick] (30) -- (40);
    \draw [->, OliveGreen,  thick] (40) -- (50);
    \draw [->, OliveGreen,  thick] (11) -- (21);
    \draw [->, OliveGreen,  thick] (21) -- (31);
    \draw [->, OliveGreen,  thick] (31) -- (41);
    \draw [->, OliveGreen,  thick] (41) -- (51);
    \draw [->, OliveGreen,  thick] (22) -- (32);
    \draw [->, OliveGreen,  thick] (32) -- (42);
    \draw [->, OliveGreen,  thick] (42) -- (52);
    \draw [->, OliveGreen,  thick] (33) -- (43);
    \draw [->, OliveGreen,  thick] (43) -- (53);
    \draw [->, OliveGreen,  thick] (44) -- (54);
    \draw [->, Magenta,  thick] (00) -- (11);
    \draw [->, Magenta,  thick] (10) -- (21);
    \draw [->, Magenta,  thick] (20) -- (31);
    \draw [->, Magenta,  thick] (30) -- (41);
    \draw [->, Magenta,  thick] (40) -- (51);
    \draw [->, Magenta,  thick] (11) -- (22);
    \draw [->, Magenta,  thick] (21) -- (32);
    \draw [->, Magenta,  thick] (31) -- (42);
    \draw [->, Magenta,  thick] (41) -- (52);
    \draw [->, Magenta,  thick] (22) -- (33);
    \draw [->, Magenta,  thick] (32) -- (43);
    \draw [->, Magenta,  thick] (42) -- (53);
    \draw [->, Magenta,  thick] (33) -- (44);
    \draw [->, Magenta,  thick] (43) -- (54);
    \draw [->, Magenta,  thick] (44) -- (55);
    \draw [->, NavyBlue,  thick] (11) -- (20);
    \draw [->, NavyBlue,  thick] (21) -- (30);
    \draw [->, NavyBlue,  thick] (31) -- (40);
    \draw [->, NavyBlue,  thick] (41) -- (50);
    \draw [->, NavyBlue,  thick] (22) -- (31);
    \draw [->, NavyBlue,  thick] (32) -- (41);
    \draw [->, NavyBlue,  thick] (42) -- (51);
    \draw [->, NavyBlue,  thick] (33) -- (42);
    \draw [->, NavyBlue,  thick] (43) -- (52);
    \draw [->, NavyBlue,  thick] (44) -- (53);
    \begin{pgfonlayer}{background}
        \draw[rounded corners=2em,line width=3em, red!75!gray!15,cap=round]
                (11.center) -- (21.center) -- (31.center)--(41.center)--(51.center);
    \end{pgfonlayer}
\end{tikzpicture}
\caption{Motzkin triangle with the second row highlighted}
\label{F:MotzkinTriangle}
\end{figure}
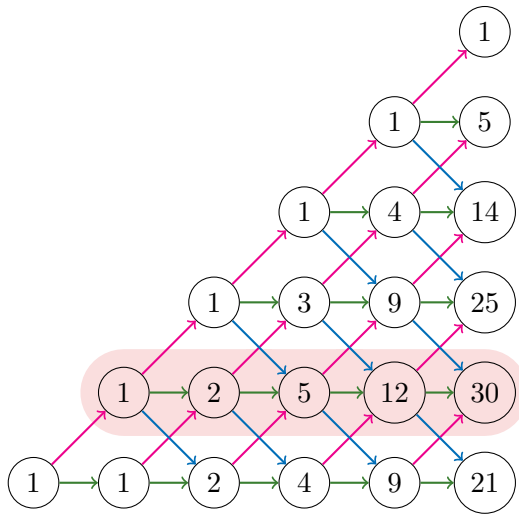
The numbers in the triangle form the sequence A026300 in~\cite{OEIS}, and the first and second rows (from bottom to top) are the sequences A001006, A002026 respectively. 
Notice in Figure~\ref{F:NJpaths6} that for all NJ paths, 
there is only one way to reach the point $(0,2)$ from the ending points $(0,3), (1,2)$, and $(2,1)$. This holds in general due to the recursion~\eqref{Eq:MotzkinTriangleDefn}. Thus, counting the number of NJ paths that end in one of these three points is equivalent to counting all NJ paths ending at $(0,2)$, or equivalently, all partial Motzkin paths ending at $(n{-}3, 1)$.  
In this way, we conclude that the number NJ paths is given by the Motzkin number $M_{n-3,1}$,
and equation~\eqref{E:Motzkin_nk} gives a formula for them.
We summarize these observations in the following theorem. 

\begin{thm}
The number 
of NJ paths for $n$ taxa
equals the Motzkin number $M_{n-3,1}$. Thus, it can be written as
\[
   \sum_{i=0}^{n-3} {n{-}3\choose i}\left[ 
   {n-i-3 \choose (n{-}i{-}2)/2}
   - {n-i-3 \choose (n{-}i)/2}
   \right] ,
\]
where, by convention, a binomial coefficient is 0 if its bottom parameter is not an integer, or if it is larger than the top parameter.
\end{thm}
%
In order to find a formula for $\Phi(n)$, the output size of the NJ algorithm, we consider \demph{weighted partial Motzkin paths} which are partial Motzkin paths with weight assignments of non-negative numbers $\{a_{k,j},\ b_{k,j},\ c_{k,j}\}$ to the steps $\{ \up{u}, \dwn{d}, \hz{h} \}$.
We interpret the weights as the multiplicity of the arrow, or equivalently, as the number of arrows in the given direction. In Figure~\ref{F:MotzkinTriangleMultipl}, we show the weight assignment to the arrows of the Motzkin triangle.
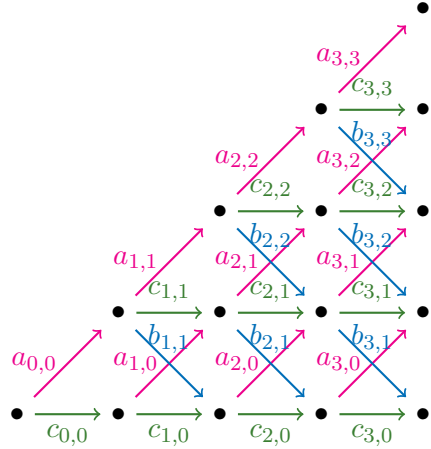
\begin{figure}[h!tb]
\begin{tikzpicture}[thick, scale=0.9]
    \node  (00) at (0,0) {$\bullet$};
    \node  (10) at (1.5,0) {$\bullet$};
    \node  (20) at (3,0) {$\bullet$};
    \node  (30) at (4.5,0) {$\bullet$};
    \node  (40) at (6,0) {$\bullet$};
    \node  (11) at (1.5,1.5) {$\bullet$};
    \node  (21) at (3,1.5) {$\bullet$};
    \node  (31) at (4.5,1.5) {$\bullet$};
    \node  (41) at (6,1.5) {$\bullet$};
    \node  (22) at (3,3) {$\bullet$};
    \node  (32) at (4.5,3) {$\bullet$};
    \node  (42) at (6,3) {$\bullet$};
    \node  (33) at (4.5,4.5) {$\bullet$};
    \node  (43) at (6,4.5) {$\bullet$};
    \node  (44) at (6,6) {$\bullet$};
    \draw [->, OliveGreen,  thick] (00) -- (10) node[midway, anchor=north]{$\hz{c_{0,0}}$};
    \draw [->, OliveGreen,  thick] (10) -- (20) node[midway, anchor=north]{$\hz{c_{1,0}}$};
    \draw [->, OliveGreen,  thick] (20) -- (30) node[midway, anchor=north]{$\hz{c_{2,0}}$};
    \draw [->, OliveGreen,  thick] (30) -- (40) node[midway, anchor=north]{$\hz{c_{3,0}}$};
    \draw [->, OliveGreen,  thick] (11) -- (21) node[midway, anchor=south]{$\hz{c_{1,1}}$};
    \draw [->, OliveGreen,  thick] (21) -- (31) node[midway, anchor=south]{$\hz{c_{2,1}}$};
    \draw [->, OliveGreen,  thick] (31) -- (41) node[midway, anchor=south]{$\hz{c_{3,1}}$};
    \draw [->, OliveGreen,  thick] (22) -- (32) node[midway, anchor=south]{$\hz{c_{2,2}}$};
    \draw [->, OliveGreen,  thick] (32) -- (42) node[midway, anchor=south]{$\hz{c_{3,2}}$};
    \draw [->, OliveGreen,  thick] (33) -- (43) node[midway, anchor=south]{$\hz{c_{3,3}}$};
    \draw [->, Magenta,  thick] (00) -- (11) node[midway, anchor=east]{$\up{a_{0,0}}$};
    \draw [->, Magenta,  thick] (10) -- (21) node[midway, anchor=east]{$\up{a_{1,0}}$};
    \draw [->, Magenta,  thick] (20) -- (31) node[midway, anchor=east]{$\up{a_{2,0}}$};
    \draw [->, Magenta,  thick] (30) -- (41) node[midway, anchor=east]{$\up{a_{3,0}}$};
    \draw [->, Magenta,  thick] (11) -- (22) node[midway, anchor=east]{$\up{a_{1,1}}$};
    \draw [->, Magenta,  thick] (21) -- (32) node[midway, anchor=east]{$\up{a_{2,1}}$};
    \draw [->, Magenta,  thick] (31) -- (42) node[midway, anchor=east]{$\up{a_{3,1}}$};
    \draw [->, Magenta,  thick] (22) -- (33) node[midway, anchor=east]{$\up{a_{2,2}}$};
    \draw [->, Magenta,  thick] (32) -- (43) node[midway, anchor=east]{$\up{a_{3,2}}$};
    \draw [->, Magenta,  thick] (33) -- (44) node[midway, anchor=east]{$\up{a_{3,3}}$};
    \draw [->, NavyBlue,  thick] (11) -- (20) node[midway, anchor=south]{$\dwn{b_{1,1}}$};
    \draw [->, NavyBlue,  thick] (21) -- (30) node[midway, anchor=south]{$\dwn{b_{2,1}}$};
    \draw [->, NavyBlue,  thick] (31) -- (40) node[midway, anchor=south]{$\dwn{b_{3,1}}$};
    \draw [->, NavyBlue,  thick] (22) -- (31) node[midway, anchor=south]{$\dwn{b_{2,2}}$};
    \draw [->, NavyBlue,  thick] (32) -- (41) node[midway, anchor=south]{$\dwn{b_{3,2}}$};
    \draw [->, NavyBlue,  thick] (33) -- (42) node[midway, anchor=south]{$\dwn{b_{3,3}}$};
    \end{tikzpicture}
    \caption{Motzkin triangle with multiplicities}
    \label{F:MotzkinTriangleMultipl}
\end{figure}
Triangles of this kind are called \demph{Motzkin triangles with multiplicities}~\cite{Lan03}. 
They generalize the Motzkin triangle from Figure~\ref{F:MotzkinTriangle},
as this is the case $a_{k,j} = b_{k,j} = c_{k,j} = 1$ for all $k,j\geq 0$.
The recursion in equation~\eqref{Eq:MotzkinTriangleDefn} generalizes to the following recursion for Motzkin triangles with multiplicities
\begin{equation}\label{Eq:WeightedMotzkinTriangleDefn}
    M_{k+1,j} = a_{k,j-1} M_{k,j-1} + c_{k,j} M_{k,j} + b_{k,j+1} M_{k,j+1}, \mbox{ for }\ k,j\geq 0.
\end{equation}

Let $t_k$ be a tree in the NJ algorithm with bough vector $(\ell,r)$, such that $k=\ell+r$.
If the next step in the NJ algorithm is a step $\alf$, one needs to choose two of the $\ell$ stems to join. There are
${\ell\choose 2}$ ways to do this. 
In a similar way, there are ${r \choose 2}$ choices for a $\bet$ step, and ${\ell \choose 1}{r \choose 1}$ for a $\gam$ step. We can use these as weights for computing the number of agglomerated trees after each step.
Thus, 
translating NJ paths to partial Motzkin paths starting at (1,1) 
by Theorem~\ref{Thm:NJ-MotzkinPaths},
and letting $s = n{-}1{-}k$, we need to assign weights in the following way:
\begin{itemize}
    \item $\up{a_{s,j}} = {{n-s-j-2 \choose 2}}$, 
    \item  $\dwn{b_{s,j}} = {{j+1 \choose 2}}$, 
    \item $\hz{c_{s,j}} = {{n-s-j-2\choose 1}{j+1\choose 1}}$,
\end{itemize}
for $0\leq s \leq n-4$ and $0\leq j\leq s$, or zero otherwise.

\begin{lem}\label{Lem:sumcoeffs}
 For $n\geq 4$ and $s \leq n-3$, the weights $a_{s,j}, b_{s,j}, c_{s,j}$ satisfy
 \begin{equation}\label{eq:sumcoeffs}
    a_{s,j}+b_{s,j}+c_{s,j} = {n-s-1\choose 2}.     
 \end{equation}
\end{lem}
\begin{proof}
From definition, equation~\eqref{eq:sumcoeffs} can be written as
\begin{align*}
    &{n{-}s{-}j{-}2\choose 2} + {j{+}1\choose 2} + {n{-}s{-}j{-}2\choose 1}{j{+}1\choose 1} = \\
    &\hspace{80pt}=  \frac{(n{-}s{-}j{-}2)(n{-}s{-}j{-}3)}{2} + \frac{(j{+}1)j}{2} + \frac{2(n{-}s{-}j{-}2)j}{2} \\
    &\hspace{80pt}= \frac{(n{-}s{-}j{-}2)^2-(n{-}s{-}j{-}2)}{2} + 
    \frac{(j{+}1)^2-(j{+}1)}{2} + \frac{2(n{-}s{-}j{-}2)j}{2}\\
    &\hspace{80pt}= \frac{(n{-}s{-}j{-}2)^2+ 2j(n{-}s{-}j{-}2) + (j{+}1)^2}{2} - \frac{(n{-}s{-}j{-}2) + (j{+}1)}{2}\\
    &\hspace{80pt}= \frac{\left((n{-}s{-}j{-}2)+(j{+}1)\right)^2}{2} - \frac{(n{-}s{-}j{-}2) + (j{+}1)}{2} \\
    &\hspace{80pt}= \frac{(n{-}s{-}1)(n{-}s{-}2)}{2} \quad = \quad {n-s-1 \choose 2}.
\end{align*}
\end{proof}

\begin{lem}\label{MainLemma}
  Let $n\geq 4$, and for $2\leq s \leq n-2$, we have
  \begin{equation}\label{eq:main_sum}
      \sum_{j=0}^{s} M_{n-s-2, j} = {s+2\choose 2} \sum_{j=0}^{s'} M_{n-s-3,j},
  \end{equation}
  with the sum in the right-hand side ending at
  $$
  s' = \begin{cases} 
    s+1 &\mbox{if }\quad  n{-}s{-}2 > \lfloor\frac{n-2}{2}\rfloor+1, \\
    s+1 &\mbox{if }\quad n{-}s{-}2 = \lfloor\frac{n-2}{2}\rfloor+1, \mbox{and } n\mbox{ is even} ,\\
    s &\mbox{if }\quad n{-}s{-}2 = \lfloor\frac{n-2}{2}\rfloor+1, \mbox{and } n\mbox{ is odd} ,\\
    s-1 & \mbox{if }\quad n{-}s{-}2 \leq \lfloor\frac{n-2}{2}\rfloor.
    \end{cases}
  $$
\end{lem}
\begin{proof}
The recursion~\eqref{Eq:WeightedMotzkinTriangleDefn} for weighted Motzkin paths writes the left-hand side of~\eqref{eq:main_sum} as
\begin{equation*}
    \begin{split}
        \sum_{j=0}^{s} M_{n-s-2, j} =& \sum_{j=0}^s \left(a_{n-s-3,j-1} M_{n-s-3,j-1} + c_{n-s-3,j} M_{n-s-3,j} + b_{n-s-3,j+1} M_{n-s-3,j+1}\right)\\
        =& (a_{n-s-3,0}+c_{n-s-3,0})M_{n-s-3,0} + \sum_{j=1}^{s-1}(a_{n-s-3,j} + c_{n-s-3,j} + b_{n-s-3,j})M_{n-s-3,j} \\
        & +  (b_{n-s-3,s}+c_{n-s-3,s})M_{n-s-3,s} + b_{n-s-3,s+1}M_{n-s-3,s+1}.
    \end{split}
\end{equation*}
Notice that, by Lemma~\ref{Lem:sumcoeffs}, we have 
$
a_{n-s-3,j} + c_{n-s-3,j} + b_{n-s-3,j} = {s+2\choose 2}
$, and by definition, $b_{n-s-3,s+1} = {s+2\choose 2}$. Lastly,
notice that $n-(n{-}s{-}3){-}s{-}2 = 1$, so
$$
(a_{n-s-3,0} + c_{n-s-3,0}) = (b_{n-s-3,s}+c_{n-s-3,s}) = {s+1\choose 2} + {1\choose 1}{s+1\choose 1} = {s+2\choose 2}.
$$
Hence, the sum in the right-hand side equals
$$
{s+2\choose 2} \sum_{j=0}^{s+1} M_{n-s-3,j}.
$$
We conclude the proof by noting that, when using the recursion~\eqref{Eq:WeightedMotzkinTriangleDefn} to write a sum of Motzkin paths of length $i{+}1$ in terms of those of length $i$, the maximum amount of summands that can appear in the sum is attained when $i=\lfloor\frac{n-2}{2}\rfloor$.
\end{proof}

Hence, we obtain the following formula for the number $\Phi(n)$. 

\begin{thm}\label{thm:CountNJCones}
For $n\geq 4$, let $\Phi(n)$ be the number of labeled agglomerated trees with $n$ leaves. Then,
$$
  \Phi(n) = {n \choose 2}{n-1\choose 2}\cdots {5\choose 2}{4\choose 2} 
  =\frac{n(n{-}1)!^2}{3\cdot 2^{n-1}}.$$
\end{thm}

\begin{proof}
From Theorem~\ref{Thm:NJ-MotzkinPaths},
computing $\Phi(n)$ is equivalent
to computing $M_{n-4,0} + M_{n-4,1} + M_{n-4,2}$ in the weighted version
of the partial Motzkin paths. From Lemma~\ref{MainLemma} we obtain
\begin{equation*}
\begin{split}
    M_{n-4,0} + M_{n-4,1} + M_{n-4,2} = & \ 
    {4\choose 2}\sum_{j=0}^3 M_{n-5,j}.
\end{split}    
\end{equation*}
We can use Lemma~\ref{MainLemma} again to compute the sum in the right-hand side. Continuing recursively we see that
$$
M_{n-4,0} + M_{n-4,1} + M_{n-4,2} = {4\choose 2}\cdots {n-1\choose 2} M_{0,0}.
$$
Defining $M_{0,0} = {n\choose 2}$ we obtain the result. 
The last equality in the theorem comes from writing ${n\choose 2}$ as $\frac{n(n-1)}{2}$ and expanding the product of binomial coefficients.
\end{proof}

We conclude this section by noticing that this formula for $\Phi(n)$ produces the sequence of the last column of Table~\ref{Tb:counting trees}; thus, the number of trees output by the NJ algorithm is $\Phi(n)/2$.
Moreover, the sequence formed by our formula in Theorem~\ref{thm:CountNJCones} does not appear in~\cite{OEIS}; however, it can be obtained from the product of consecutive binomial coefficients.
This product forms the sequence A006472, which counts the number of ranked trees.
Thus, our sequence is obtained from A006472 starting from $n=4$ by dividing each element by 3. This suggest a connection between agglomerated trees and ranked trees~\cite{DW13}.



\section{Estimated volumes}\label{S:volumes}

The Neighbor-Joining and UPGMA algorithms take dissimilarity maps as inputs and return trees with an additive dissimilarity map.  The decisions in both algorithms are based on linear inequalities that divide $\R^{n \choose 2}$ into half-spaces.
To elucidate, label the coordinates of $\R^{n \choose 2}$ with the 2-element subsets of $[n]=\{1,2,\ldots,n\}$, so that the symmetric matrix $D_n$ is a point in $\R^{n \choose 2}$.
Recall, for instance, the example 
in the beginning of Section~\ref{S:combinatoricsNJ}.
The two matrices $D$ and $D'$ from~\eqref{Eqn:DandD'Dissimilarity}
are both in correspondence with the tree in Figure~\ref{F:example1}.
For the matrix $D$, the tree returned by the NJ algorithm can be represented as $((d,(a,b)),c,e)$ in Newick format.
This means that the first cherry formed by the algorithm was $(a,b)$. 
For this to hold true, the $Q$-criterion of the matrix $D$ must satisfy the following 9 inequalities:
\begin{equation}\label{Eq:halfspaces}
Q(a,b) \leq Q(a,c), Q(a,d), 
\ldots, Q(c,e), Q(d,e).
\end{equation}
Since each entry $Q(i,j)$ is a linear equation on the entries of $D$, these equations form semialgebraic sets in $\R^{15}$. All input matrices $D_5$ lying in the intersection of the 9 half-spaces listed in~\eqref{Eq:halfspaces} will return 
two possible agglomerated trees.
In general, these halfspaces form a polyhedral cone in $\R^{n\choose 2}$, and every input that satisfies them will 
receive the same output.
Thus, we identify this cone with the two labeled agglomerated trees.  
For $D'$ the NJ algorithm associates the trees $((d,(c,e)),a,b)$ and $((a,b),(c,e),d)$. Thus, the 9 inequalities for $D'$ are not the same as those for $D$. Therefore, the trees lie on different cones, even though they are the same topologically.

Eickmeyer et.al.~\cite{EHLY08,EY07} studied the defining inequalities of the NJ cones for small taxa.  In other words, the cones were described using hyperplanes, also known as an $H$-representation.  
To give a a full combinatorial description, one requires a $V$-representation: to describe the vectors for the extreme rays of the cones.
In contrast, for the UPGMA algorithm~\cite{SM58}, a $V$-representation for the corresponding polyhedral cones was given in \cite{DS13}. 

\begin{rem}\label{rem:bias}
Ideally, a phylogenetic method would depend only on the quality of the data it receives as input, so we would expect it to return each tree with equal probability. This would imply that the algorithm has no bias towards a specific type of tree.
\end{rem}

One can estimate this probability by measuring the \demph{spherical volume} of the polyhedral cones, which is the proportional volume when intersecting the cones with the unit sphere in $\R^{n \choose 2}   $.
A bias in the algorithm towards some trees can be detected when the spherical volume is not roughly the same for each tree.
Computational methods based in this paradigm are used to evaluate the performance of NJ as a heuristic for LSP in \cite{DS14} and of NJ as a heuristic for BME in \cite{EHLY08}.  
Davidson and Sullivant~\cite{DS13} approximated the spherical volume of the UPGMA cones finding cones with considerable less volume than others, unveiling a bias of the algorithm towards balanced trees (rooted trees with two children of the root defining subtrees with a similar amount of leaves).

We studied the approximated spherical volume of the NJ cones, by analyzing how simulated data lies in the cones.
We implemented the NJ algorithm in Mathematica~\cite{Math}.
Our implementation takes as input a dissimilarity map, represented as a symmetric real matrix, and returns the agglomerated tree as well as all the agglomeration steps.
Our software, as well as the results of the computations discussed in the rest of the paper, is available at~\cite{WWWNJCones}.

\subsection{Simulation results} 
In order to approximate the volume of the polyhedral cones, 
for $n=4,\ldots, 8$, we uniformly sampled 
1,000,000
data matrices lying inside the unit sphere and the positive orthant of $ \R^{n \choose 2}   $.  For each, our software computed its corresponding agglomerated tree.
We counted the number of matrices that were associated to each agglomerated tree.

Table~\ref{tbl:4taxaRegular} displays our results for 4 taxa. The first column and third represents the  agglomerated trees in Newick format, each row has the two trees identified by the NJ algorithm. The second and forth columns show the percentage of the 
million random instances output by NJ. 
Note that in this case, the NJ algorithm should output 3 trees, due to Lemma~\ref{L:double_minima}. 
Hence, we should expect that three of the six trees should not had been observed. 
However, the three trees in the first column accumulated only 86.7771\% of the random input matrices, and the remaining instances were equally distributed among the other three cones.
We noted this behavior is due to the way Mathematica resolves the double minimum conflict.
%
\begin{table}[htb]
\caption{Percentage of the 1,000,000 samples in each agglomerated tree with 4 taxa.}
\label{tbl:4taxaRegular}
\begin{tabular}{|c|r||c|r|}
\hline
\multicolumn{4}{|c|}{$4$ taxa}                       \\ \hline
{\bf Tree}                           & {\bf Percentage} & {\bf Tree}     & {\bf Percentage} \\ \hline
$((ab)cd)$     & 29.1610              & 
$((cd)ab)$  & 4.1150               \\ 
$((ac)bd)$ & 28.9944              & 
$((bd)ac)$  & 4.4149              \\ 
$((ad)bc)$  & 28.6217              & 
$((bc)ad)$  & 4.6930               \\ \hline
\end{tabular}
\end{table}
%

Each row of Table~\ref{tbl:4taxaRegular}
reports the observed instances in each polyhedral cone, so adding the corresponding percentage for each row gives the distribution of the instances lying in each cone. Here, we noted that the row sum is around 33.333\%, showing that each instance is equally distributed among the three cones.  In light of this fact the NJ algorithm does not have a bias resulting from variation in cone size in the case of 4 taxa.  We shall see this is not the case for larger numbers of taxa.  In the case where $n \geq 5$ the last decision step of NJ can be studied in $\mathbb{R}^{6}$, but 
at least one of $\{a,b,c,d\}$ is a bouquet of size at least 2.  

We remark that all implementations have to resolve the decision of the tie in the last step, and the original paper of Saitou and Nei~\cite{NJ87} noticed the tie in equations (11b) and (11c), suggesting to resolve a tie
between $(ab)$ and $(cd)$ by choosing to agglomerate the one with the smallest index (lexicographic), in this case $(ab)$.  
We call this a \demph{representative bias},
and we explored two ways to correct it. 

\subsection{Representative bias correction} 
There could be many ways to resolve the tie in the last step of the algorithm.
One could choose a distinguished set of pairs to resolve the clash of minima.
For instance, we could always choose to agglomerate the pair involving the lowest (or highest) index. This choice would make it so that all instances in Table~\ref{tbl:4taxaRegular} would
be equally distributed between the trees $((1,2),3,4), ((1,3),2,4)$, and $((1,4),2,3)$, and the other three would not observe any instances. This could be desirable for practical reasons. 
However, this convenient choice has an impact on the  interpretation of the output, 
as we would be choosing beforehand a step in which the algorithm joins specific taxa. 
Thus, we explored the impact of some different choices.

Besides from the lexicographic choice, another natural way to correct the ambiguity in the last step of the NJ algorithm is to 
choose uniformly at random one of the two coinciding agglomerations. We repeated the experiment with this correction in our software, and found this choice equally distributes the number of instances observed in each agglomerated tree. We refer to this representative bias correction just as \demph{Uniform}.
Table~\ref{tbl:4taxaUniform} displays the results for 4 taxa after this correction.
\begin{table}[htb]
\begin{tabular}{|c|r||c|r|}
\hline
\multicolumn{4}{|c|}{Uniform bias correction}          \\ \hline
\multicolumn{4}{|c|}{$4$ taxa} \\ \hline
{\bf Tree}                           & {\bf Percentage} & {\bf Tree}     & {\bf Percentage} \\ \hline
$((ab)cd)$     & 16.8740              & 
$((cd)ab)$  & 16.4977               \\ 
$((ac)bd)$ & 16.7543              & 
$((bd)ac)$  & 16.5304              \\ 
$((ad)bc)$  & 16.5622              & 
 $((bc)ad)$  & 16.7814               \\ \hline
\end{tabular}
\caption{Percentage of 1,000,000 samples found in each agglomerated tree.}
\label{tbl:4taxaUniform}
\end{table}

Lastly, we also consider another 
representative bias correction method, that considers previous agglomeration information. As noticed, the last step of the algorithm could have to decide between joining two sets of taxa $A, B$, over other two $C, D$, and we could take biological information in consideration for making this choice. For instance, one could choose the representative that joins the pair containing more taxa: we call this representative bias correction method \demph{Baggage}. If, in the last step, the NJ algorithm has to decide between an $\alf$ or a $\bet$ step, Baggage would chose $\bet$ over $\alf$ as it is joining two bouquets having at least 2 leaves in each, as opposed to $\alf$ that is joining two stems, i.e., two leaves.
Formally, we think of the NJ algorithm as joining sets of taxa, starting from a list of $n$ singletons, and stopping until we obtain three taxa sets. 
For a taxa set $A$ we let $|A|$ denote the size of $A$, which is the number of taxa it contains (alternatively, we can think of $|A|$ as the size of a bough). Then, Baggage chooses to join taxa sets $A$ and $B$ over $C$ and $D$, if $|A|+|B| > |C|+|D|$. 
If both pairs of sets have the same taxa, so  $|A|+|B| $ equals $ |C|+|D|$, we could use the uniform method to solve this case. We note that making the opposite choice --joining a pair $A, B$ that is more balanced-- could produce a representative bias towards topologically balanced trees similar to the one discovered in the UPGMA algorithm~\cite{DS14} and observed in NJ in some cases for $n \leq 14$. 
%

For $n=4$, the Baggage correction and the Uniform coincide, but this is not the case for $n\geq 5$. We show in Table~\ref{tbl:5taxaALL} the comparison for 5 taxa of our 
implementation without representative bias correction (column labeled Regular), versus the Uniform, and Baggage bias correction methods. 
To interpret the results, in each row we display the percentage obtained for the two trees identified by the NJ algorithm. We notice that for every row, the sum of the entries of each representative bias correction for both trees is around 3.3333\%. Thus, 
from Remark~\ref{rem:bias} the NJ method have no bias. However, if the NJ method had no representative bias, each agglomerated tree should be returned with probability 1.666\%, which is very close to what is obtained in the Uniform method.
%
\begin{longtable}{|c|c|c|c||c|c|c|c|}
\hline
\multicolumn{8}{|c|}{$5$ taxa}                       \\ \hline
{\bf Tree} & {\bf Regular} & {\bf Uniform} & {\bf Baggage} &
{\bf Tree} & {\bf Regular} & {\bf Uniform} & {\bf Baggage} \\ \hline
$((bd)(ac)e)$
&  2.9818
&  1.6625
&  0.3924
&
$((e(ac))bd)$
&  0.3349
&  1.6562
&  2.9713
\\
$((ad)(bc)e)$
&  3.0313
&  1.6897
&  0.3798
&
$((e(bc))ad)$
&  0.3401
&  1.6503
&  2.9405
\\ 
$((cd)(ab)e)$
&  3.0006
&  1.7055
&  0.3833
&
$((e(ab))cd)$
&  0.3384
&  1.6309
&  2.9239
\\ 
$((bc)(ad)e)$
&  2.9995
&  1.6981
&  0.3864
&
$((e(ad))bc)$
&  0.3362
&  1.6335
&  2.9338
\\ 
$((ae)(cd)b)$
&  2.9926
&  1.6741
&  0.3619
&
$((b(cd))ae)$
&  0.3553
&  1.6657
&  2.9849
\\ 
$((ab)(cd)e)$
&  3.0014
&  1.6839
&  0.3991
&
$((e(cd))ab)$
&  0.3369
&  1.6553
&  2.9486
\\ 
$((bc)(ae)d)$
&  3.0251
&  1.6716
&  0.3839
&
$((d(ae))bc)$
&  0.3347
&  1.6510
&  2.9594
\\ 
$((bd)(ae)c)$
&  2.9969
&  1.6679
&  0.3538
&
$((c(ae))bd)$
&  0.3689
&  1.6716
&  2.9748
\\ 
$((ce)(ab)d)$
&  2.9730
&  1.6668
&  0.3619
&
$((d(ab))ce)$
&  0.3524
&  1.6545
&  2.9693
\\ 
$((a(bc))de)$
&  2.9880
&  1.6459
&  2.9704
&
$((de)(bc)a)$
&  0.3717
&  1.6588
&  0.3558
\\ 
$((ae)(bd)c)$
&  2.9930
&  1.6647
&  0.3560
&
$((c(bd))ae)$
&  0.3581
&  1.6634
&  2.9807
\\ 
$((be)(ad)c)$
&  2.9162
&  1.6854
&  0.3611
&
$((c(ad))be)$
&  0.3510
&  1.6643
&  2.9518
\\ 
$((b(ac))de)$
&  2.9415
&  1.6558
&  2.9923
&
$((de)(ac)b)$
&  0.3825
&  1.6723
&  0.3709
\\ 
$((ae)(bc)d)$
&  2.9802
&  1.6918
&  0.3610
&
$((d(bc))ae)$
&  0.3489
&  1.6358
&  2.9594
\\ 
$((ac)(be)d)$
&  3.0146
&  1.7032
&  0.3772
&
$((d(be))ac)$
&  0.3459
&  1.6668
&  2.9248
\\ 
$((c(ab))de)$
&  2.9679
&  1.6371
&  2.9771
&
$((de)(ab)c)$
&  0.3800
&  1.6696
&  0.3582
\\ 
$((a(cd))be)$
&  2.9911
&  1.6576
&  2.9827
&
$((be)(cd)a)$
&  0.3826
&  1.6674
&  0.3530
\\ 
$((a(be))cd)$
&  2.9537
&  1.6768
&  2.9807
&
$((cd)(be)a)$
&  0.3776
&  1.6902
&  0.3647
\\
$((a(ce))bd)$
&  2.9624
&  1.6526
&  2.9743
&
$((bd)(ce)a)$
&  0.3668
&  1.6915
&  0.3623
\\ 
$((ad)(ce)b)$
&  2.9964
&  1.6863
&  0.3570
&
$((b(ce))ad)$
&  0.3746
&  1.6500
&  2.9664
\\ 
$((ad)(be)c)$
&  2.9102
&  1.6850
&  0.3823
&
$((c(be))ad)$
&  0.3603
&  1.6433
&  2.9816
\\ 
$((be)(ac)d)$
&  2.9823
&  1.6898
&  0.3651
&
$((d(ac))be)$
&  0.3610
&  1.6402
&  2.9838
\\ 
$((a(de))bc)$
&  2.9475
&  1.6272
&  2.9928
&
$((bc)(de)a)$
&  0.3823
&  1.6826
&  0.3604
\\ 
$((b(ad))ce)$
&  2.9571
&  1.6316
&  2.9331
&
$((ce)(ad)b)$
&  0.3788
&  1.6946
&  0.3600
\\ 
$((ac)(de)b)$
&  2.9495
&  1.6927
&  0.3682
&
$((b(de))ac)$
&  0.3603
&  1.6336
&  2.9816
\\ 
$((ab)(ce)d)$
&  2.9766
&  1.6883
&  0.3851
&
$((d(ce))ab)$
&  0.3364
&  1.6620
&  2.9377
\\ 
$((ab)(de)c)$
&  2.9929
&  1.6864
&  0.3712
&
$((c(de))ab)$
&  0.3298
&  1.6542
&  2.9533
\\ 
$((b(ae))cd)$
&  2.9349
&  1.6516
&  2.9651
&
$((cd)(ae)b)$
&  0.3642
&  1.6660
&  0.3585
\\ 
$((a(bd))ce)$
&  2.9472
&  1.6508
&  2.9784
&
$((ce)(bd)a)$
&  0.3736
&  1.6882
&  0.3614
\\ 
$((ac)(bd)e)$
&  2.9800
&  1.7034
&  0.3772
&
$((e(bd))ac)$
&  0.3304
&  1.6521
&  2.9564
\\ 
\hline
\caption{Fraction of 1,000,000 samples in each tree for different bias corrections.}
\label{tbl:5taxaALL}
\end{longtable}
In Table~\ref{tbl:5taxaALL} we see again that our original implementation had an unexpected bias due to the way Mathematica sorted the double minima conflict in the last step. The values in each row corresponding to the column Regular should be zero for all the trees in the columns of the right-hand side, according to the lexicographic choice.
The Uniform method seems to correct the representative bias completely, making each agglomerated tree occur with equal probability. Lastly, the Baggage method creates a representative bias towards agglomerated trees formed with less $\alf$ steps. 
There could be biological reasons to correct the bias to agree with what is observed, such as high confidence in the appearance of a subtree within a larger taxon set.   
We realized these computations for $n=4,\ldots, 8$ taxa, obtaining similar results to those displayed in Table~\ref{tbl:5taxaALL} and are available at~\cite{WWWNJCones}.

We conclude by remarking that 
it is left to understand other reasonable ways to explore methods of bias correction, and understand their combinatorial implications.
Knowing about different representative biases that can arise in the algorithm
not only 
helps to develop other uses for this agglomerative method,
but 
also to aid in the study of bias from other heuristics for LSP or BME as well as phylogenetic pipelines that contain NJ as a component. 

\bibliographystyle{siam}
\bibliography{refs}

\begin{thebibliography}{10}

\bibitem{Aig98}
{\sc M.~Aigner}, {\em Motzkin numbers}, European J. Combin., 19 (1998),
  pp.~663--675.

\bibitem{Aig99}
\leavevmode\vrule height 2pt depth -1.6pt width 23pt, {\em Catalan-like numbers
  and determinants}, J. Combin. Theory Ser. A, 87 (1999), pp.~33--51.

\bibitem{Bon15}
{\sc M.~B\'{o}na}, ed., {\em Handbook of enumerative combinatorics}, Discrete
  Mathematics and its Applications (Boca Raton), CRC Press, Boca Raton, FL,
  2015.

\bibitem{Bry05}
{\sc D.~Bryant}, {\em On the uniqueness of the selection criterion in
  neighbor-joining}, J. Classification, 22 (2005), pp.~3--15.

\bibitem{WWWNJCones}
{\sc R.~Davidson and A.~Mart\'in~del {C}ampo}, {\em Supplementary materials}.
\newblock \emph{“{C}ombinatorial and computational investigations of
  {N}eighbor-{J}oining bias”}, 2020.
\newblock https://www.cimat.mx/\~{}abraham.mc/NJCones/NJCones.html.

\bibitem{DS13}
{\sc R.~Davidson and S.~Sullivant}, {\em Polyhedral combinatorics of {UPGMA}
  cones}, Adv. in Appl. Math., 50 (2013), pp.~327--338.

\bibitem{DS14}
{\sc R.~Davidson and S.~Sullivant}, {\em Distance-based phylogenetic methods
  around a polytomy.}, IEEE/ACM Transactions on Computational Biology and
  Bioinformatics, 11 (2014), pp.~325--335.

\bibitem{Day87}
{\sc W.~Day}, {\em Computational complexity of inferring phylogenies from
  dissimilarity matrices.}, Bulletin of Mathematical Biology, 49 (1987),
  pp.~461--467.

\bibitem{DG04}
{\sc R.~Desper and O.~Gascuel}, {\em {Theoretical Foundation of the Balanced
  Minimum Evolution Method of Phylogenetic Inference and Its Relationship to
  Weighted Least-Squares Tree Fitting}}, Molecular Biology and Evolution, 21
  (2004), pp.~587--598.

\bibitem{DSh02}
{\sc E.~Deutsch and L.~W. Shapiro}, {\em A bijection between ordered trees and
  2-{M}otzkin paths and its many consequences}, Discrete Math., 256 (2002),
  pp.~655--670.
\newblock LaCIM 2000 Conference on Combinatorics, Computer Science and
  Applications (Montreal, QC).

\bibitem{DW13}
{\sc F.~Disanto and T.~Wiehe}, {\em Exact enumeration of cherries and
  pitchforks in ranked trees under the coalescent model}, Mathematical
  Biosciences, 242 (2013), pp.~195 -- 200.

\bibitem{EHLY08}
{\sc K.~Eickmeyer, P.~Huggins, L.~Pachter, and R.~Yoshida}, {\em On the
  optimality of the neighbor-joining algorithm}, Algorithms for Molecular
  Biology, 3 (2008), p.~5.

\bibitem{EY07}
{\sc K.~Eickmeyer and R.~Yoshida}, {\em The geometry of the neighbor-joining
  algorithm for small trees}, in Algebraic Biology, K.~Horimoto,
  G.~Regensburger, M.~Rosenkranz, and H.~Yoshida, eds., Berlin, Heidelberg,
  2008, Springer Berlin Heidelberg, pp.~81--95.

\bibitem{fels04}
{\sc J.~Felsenstein}, {\em Inferring Phylogenies}, Sinauer Associates, Inc.,
  Sunderland, MA, U.S.A., 2004.

\bibitem{GS06}
{\sc O.~Gascuel and M.~Steel}, {\em {Neighbor-Joining Revealed}}, Molecular
  Biology and Evolution, 23 (2006), pp.~1997--2000.

\bibitem{Att97}
{\sc T.~Jiang and D.~Lee}, eds., {\em The performance of neighbor-joining
  algorithms of phylogeny reconstruction}, vol.~1276, Berlin, Heidelberg, 1997,
  COCOON, Springer.
\newblock Computing and Combinatorics.

\bibitem{Lan03}
{\sc S.~K. Lando}, {\em Lectures on generating functions}, vol.~23 of Student
  Mathematical Library, American Mathematical Society, Providence, RI, 2003.
\newblock Translated from the 2002 Russian original by the author.

\bibitem{SNPhylo}
{\sc T.~Lee, H.~Guo, X.~Wang, C.~Kim, and A.~Paterson}, {\em Snphylo: a
  pipeline to construct a phylogenetic tree from huge snp data.}, BMC Genomics,
  15 (2014).

\bibitem{NJst}
{\sc L.~Liu and L.~Yu}, {\em Estimating species trees from unrooted gene
  trees.}, Systematic Biology, 60 (2011), pp.~661--667.

\bibitem{MOPT10}
{\sc V.~R. Meshkov, A.~V. Omelchenko, M.~I. Petrov, and E.~A. Tropp}, {\em Dyck
  and {M}otzkin triangles with multiplicities}, Mosc. Math. J., 10 (2010),
  pp.~611--628, 662.

\bibitem{MLP09}
{\sc R.~Mihaescu, D.~Levy, , and L.~Pachter}, {\em Why neighbor-joining
  works.}, Algorithmica, 54 (2009), pp.~1--24.

\bibitem{OVdJ15}
{\sc R.~Oste and J.~Van~der Jeugt}, {\em Motzkin paths, {M}otzkin polynomials
  and recurrence relations}, Electron. J. Combin., 22 (2015), pp.~Paper 2.8,
  19.

\bibitem{NJ87}
{\sc N.~Saitou and M.~Nei}, {\em The neighbor-joining method: a new method for
  reconstructing phylogenetic trees.}, Molecular Biology and Evolution, 4
  (1987), pp.~406--425.

\bibitem{Phyl}
{\sc C.~Semple and M.~Steel}, {\em Phylogenetics}, Oxford lecture series in
  mathematics and its applications, Oxford University Press, 2003.

\bibitem{OEIS}
{\sc N.~J.~A. Sloane}, {\em The on-line encyclopedia of integer sequences}.
\newblock OEIS Foundation Inc., 2020.
\newblock http://oeis.org.

\bibitem{SM58}
{\sc R.~R. Sokal and C.~D. Michener}, {\em A statistical method of evaluating
  systematic relationships}, University of Kansas Science Bulletin, 38 (1958),
  pp.~1409--1438.

\bibitem{SS04}
{\sc D.~Speyer and B.~Sturmfels}, {\em The tropical grassmannian}, Advances in
  Geometry, 4 (2003), pp.~389--411.

\bibitem{SK88}
{\sc J.~A. Studier and K.~J. Keppler}, {\em {A note on the neighbor-joining
  algorithm of Saitou and Nei.}}, Molecular Biology and Evolution, 5 (1988),
  pp.~729--731.

\bibitem{LiveNJ}
{\sc G.~Telles, G.~Araújo, M.~Walter, M.~Brigido, and N.~Almeida}, {\em Live
  neighbor-joining.}, BMC Bioinformatics, 19 (2018).

\bibitem{War17}
{\sc T.~Warnow}, {\em Computational phylogenetics: An introduction to designing
  methods for phylogeny estimation}, Cambridge University Press, 2017.

\bibitem{Math}
{\sc I.~Wolfram~Research}, {\em Mathematica, {V}ersion 12.1}.
\newblock https://www.wolfram.com/mathematica.
\newblock Champaign, IL, 2020.

\end{thebibliography}

\end{document}